\documentclass{amsart}
\usepackage{amsfonts}
\usepackage{soul}
\usepackage{color}
\usepackage[english]{babel}
\newtheorem{thm}{Theorem}[section]
\newtheorem{cor}[thm]{Corollary}
\newtheorem{lem}[thm]{Lemma}
\newtheorem{prop}[thm]{Proposition}
\newtheorem{defn}[thm]{Definition}

\begin{document}
\title[Solvable Leibniz algebras with triangular nilradicals] {Solvable Leibniz algebras with triangular nilradicals }
\author{ I.A. Karimjanov, A.Kh. Khudoyberdiyev, B. A. Omirov }
\address{[ I.A. Karimjanov --- A.Kh. Khudoyberdiyev --- B.A. Omirov] National University of Uzbekistan, Institute of Mathematics, 29, Do'rmon yo'li street., 100125, Tashkent (Uzbekistan)} \email{iqboli@gmail.com --- khabror@mail.ru --- omirovb@mail.ru}


\begin{abstract}
In this paper the description of solvable Lie algebras with
triangular nilradicals is extended to Leibniz algebras. It is
proven that the matrices of the left and right operators on
elements of Leibniz algebra have upper triangular forms. We
establish that solvable Leibniz algebra of a maximal possible
dimension with a given triangular nilradical is a Lie algebra.
Furthermore, solvable Leibniz algebras with triangular nilradicals
of low dimensions are classified.
\end{abstract}

\maketitle



\section{Introduction}

Leibniz algebras were introduced at the beginning of the 90s of
the past century by J.-L. Loday \cite{Lod}. They are a
generalization of well-known Lie algebras, which admit a
remarkable property that an operator of right multiplication is a
derivation.

From the classical theory of Lie algebras it is well known that
the study of finite-dimensional Lie algebras was reduced to the
nilpotent ones \cite{Mal}, \cite{Mub}. In the Leibniz algebra case
there is an analogue of Levi's theorem \cite{Barn}. Namely, the
decomposition of a Leibniz algebra into a semidirect sum of its
solvable radical and a semisimple Lie algebra is obtained. The
semisimple part can be described from simple Lie ideals (see
\cite{Jac}) and therefore, the main focus is to study the solvable
radical.

The analysis of several works devoted to the study of solvable Lie
algebras (for example \cite{An1,An2,Ngomo,Wint,Wang}, where
solvable Lie algebras with various types of nilradical were
studied, such as naturally graded filiform and quasi-filiform
algebras, abelian, triangular, etc.) shows that we can also apply
similar methods to solvable Leibniz algebras with a given
nilradical. In fact, any solvable Lie algebra can be represented
as an algebraic sum of a nilradical and its complimentary vector
space.  Mubarakdjanov proposed a method, which claims that the
dimension of the complimentary vector space does not exceed the
number of nil-independent derivations of the nilradical
\cite{Mub}. Extension of this method to Leibniz algebras is shown
in \cite{Cas1}. Usage of this method yields a classification of
solvable Leibniz algebras with given nilradicals in
\cite{Cas1,Cas2,CanKhud,Lad,Lin}.

In this article we present the description of solvable Leibniz
algebras whose nilradical is a Lie algebra of upper triangular
matrices. Since in the work \cite{Wint} solvable Lie algebras with
triangular nilradical are studied, we reduce our study to non-Lie
Leibniz algebras.

Recall, that in \cite{Wint} solvable Lie algebras with triangular
nil-radicals of minimum and maximum  possible dimensions were
described. Moreover, uniqueness of a Lie algebra of maximal
possible dimension with a given triangular nilradical is
established.

In order to realize the goal of our study we organize the paper as
follows. In Section 2 we give the necessary preliminary results.
Section 3 is devoted to the description of a finite-dimensional
solvable Leibniz algebras with upper triangular nilradical. We
establish that such Leibniz algebras of minimum and maximum
possible dimensions are Lie algebras.  Finally, in Section 4 we
present complete description of the results of Section 3 in low
dimensions.

Throughout the paper we consider finite-dimensional vector spaces
and algebras over the field $\mathbb{C}$. Moreover, in the
multiplication table of an algebra omitted products are assumed to
be zero and if it is not stated otherwise, we will consider
non-nilpotent solvable algebras.

\section{Preliminaries}

In this section we give the basic concepts and the results used in
the studying of Leibniz algebras with triangular nilradicals.

\begin{defn} An algebra $(L,[-,-])$ over a field $F$ is called a Leibniz algebra if for any $x,y,z\in L$ the so-called Leibniz identity
$$[x,[y,z]]=[[x,y],z] - [[x,z],y]$$  holds.
\end{defn}

Every Lie algebra is a Leibniz algebra, but the bracket in the
Leibniz algebra does not possess a skew-symmetric property.

\begin{defn}For a given Leibniz algebra $L$  the sequences of two-sided ideals defined recursively as follows:
$$L^1=L, \ L^{k+1}=[L^k,L],  \ k \geq 1, \quad\quad
L^{[1]}=L, \ L^{[s+1]}=[L^{[s]},L^{[s]}], \ s \geq 1.
$$
are called the lower central and the derived series of $L$,
respectively.
\end{defn}

\begin{defn} A Leibniz algebra $L$ is said to be
nilpotent (respectively, solvable), if there exists $n\in\mathbb N$ ($m\in\mathbb N$) such that $L^{n}=0$ (respectively, $L^{[m]}=0$).
\end{defn}

It is easy to see that the sum of any two nilpotent ideals is
nilpotent. Therefore the maximal nilpotent ideal always exists.

 \begin{defn}
  The maximal nilpotent ideal of a Leibniz algebra is said to be a nilradical of the algebra.
  \end{defn}

Recall, that a linear map $d : L \rightarrow L$ of a Leibniz
algebra $L$ is called a derivation if for all $x, y \in L$ the
following  condition holds: $$d([x,y])=[d(x),y] + [x, d(y)].$$

For a given element $x$ of a Leibniz algebra $L$ we consider a
right multiplication operators $R_x : L \to L$ defined by
$R_x(y)=[y,x], \forall y \in L$ and the left multiplication
operators $L_x : L \to L$ defined by $L_x(y)=[x,y], \forall y \in
L$. It is easy to check that operator $R_x$ is a derivation. This
kind of derivations are called {\it inner derivations}.

Linear maps $f_1,...,f_k$ are called {\it nil-independent}, if
$$\alpha_1f_1+\alpha_2f_2+...+\alpha_kf_k$$ is not nilpotent for all values $\alpha_i,$ except simultaneously zero.

Let $R$ be a solvable Leibniz algebra with a nilradical $N$. We
denote by $Q$ the complementary vector space  of the nilradical
$N$ in the algebra $R.$

\begin{prop}\cite{Cas1} Let $R$ be a solvable Leibniz algebra and $N$ its nilradical. Then the
dimension of the complementary vector space $Q$ is not greater
than the maximal number of nil-independent derivations of $N$.
\end{prop}

Let us consider a finite-dimensional Lie algebra  $T(n)$ of
upper-triangular matrices with $n\geq3$ over the field of complex
numbers. The products of the basis elements $\{N_{ij} \ | \ 1\leq
i<j\leq n \}$ of $T(n)$, where $N_{ij}$ is a matrix with the only
non-zero entry at $i$-th row and $j$-th column equal to 1, can be
computed by
$$[N_{ij},N_{kl}]=\delta_{jk}N_{il}-\delta_{il}N_{kj}.$$

For a natural number $f$ let $G(n,f)$ be a set of solvable Lie
algebras of dimension $\frac{1}{2}n(n-1)+f$ with nilradical
$T(n).$ Let $Q=<X^1,X^2,\dots,X^f>,$ where $Q$ is the
complementary vector space of the nilradical $T(n)$ to an algebra
from $G(n,f)$.

Denote

\begin{equation}\label{eq1}
[N_{ij},X^{\alpha}]=\sum\limits_{1\leq q-p<n}a^{\alpha}_{ij,pq} N_{pq}, \quad  [X^{\alpha}, N_{ij}]=\sum\limits_{1\leq q-p<n}b^{\alpha}_{ij,pq} N_{pq}, \quad [X^{\alpha},X^{\beta}]=\sum\limits_{1\leq q-p<n}\sigma^{\alpha\beta}_{pq} N_{pq},
\end{equation}
where $1\leq\alpha,\beta\leq f$ and $a^{\alpha}_{ij,pq},
b^{\alpha}_{ij,pq}, \sigma^{\alpha\beta}_{pq}\in \mathbb{C}, \
p<q\leq n.$

Let $N$ be a vector column $(N_{12} \ N_{23}\ \dots \ N_{(n-1)n} \
N_{13} N_{24}\dots \ N_{(n-2)n}\dots \ N_{1n})^T$   then we have
$$ R_{X^{\alpha}}(N)=A^{\alpha}N, \quad
L_{X^{\alpha}}(N)=B^{\alpha}N,$$ where
$A^{\alpha}=(a^{\alpha}_{ij,pq})$ and
$B^{\alpha}=(b^{\alpha}_{ij,pq}), \ 1\leq i<j\leq n, \ 1\leq
p<q\leq n$  are $\frac12 n(n-1)\times \frac12 n(n-1)$ complex
matrices.

The following lemma provides some information about the structure
matrices above.

\begin{lem}\cite{Wint} \label{lem2.4}The structure matrices   $A^{\alpha}=(a_{ij,pq}^{\alpha}), \ 1\leq i<j\leq n, \ 1\leq p<q\leq n$ have the following properties:

$(i)$ They are upper triangular;

$(ii)$ The only off-diagonal matrix elements that do not vanish identically and cannot be annuled by a redefinition of the elements $X^\alpha$ are: $$a_{12,2n}^{\alpha}, \ \ a_{i(i+1),1n}^{\alpha} \ (2\leq i\leq n-2), \ \ a_{(n-1)n,1(n-1)}^{\alpha},$$

$(iii)$ The diagonal elements  $a_{i(i+1),i(i+1)}^{\alpha}, \
1\leq i \leq n-1$ are free to vary. The other diagonal elements
satisfy
$$a_{ik,ik}^{\alpha}=\sum\limits_{p=i}^{k-1}a_{p(p+1),p(p+1)}^{\alpha},
\ k>i+1.$$
\end{lem}

\begin{lem} \cite{Wint} The maximal number of non-nilpotent elements is
$$f_{max} = n - 1.$$
\end{lem}

\section{Main result}

We denote by $L(n,f)$ a set of all non-nilpotent solvable  Leibniz algebras with nilradical $T(n)$ and a complementary vector space $<X^1,X^2,...,X^f>.$

Using notations (\ref{eq1}) we have  $$
R_{X^{\alpha}}(N)=A^{\alpha}N, \ L_{X^{\alpha}}(N)=B^{\alpha}N,$$
where $A^{\alpha}=(a^{\alpha}_{ij,pq})$ è
$B^{\alpha}=(b^{\alpha}_{ij,pq}), \ 1\leq i<j\leq n, \ 1\leq
p<q\leq n.$

Since the proof of the assertions concerning the elements of the
matrix $A^{\alpha}$ in Lemma \ref{lem2.4} uses only the property
of derivation, one can check that it obviously extends to our case
of Leibniz algebras. For the matrix $B^{\alpha}$ however, we have
the next result.

\begin{lem} \label{lem3.1} The following relations hold:
$$b^\alpha_{ij,pq}=-a^\alpha_{ij,pq}, \quad i+1 < j, \quad (p,q)\neq (1,n)$$

\end{lem}
\begin{proof} From Lemma \ref{lem2.4} we conclude

$$\begin{array}{lll}
[N_{12},X^\alpha] & = a^\alpha_{12,12}N_{12}+a^\alpha_{12,2n}N_{2n},&\\[1mm]
[N_{i(i+1)},X^\alpha] & = a^{\alpha}_{i(i+1),i(i+1)}N_{i(i+1)}+a^{\alpha}_{i(i+1),1n} N_{1n}, & 2\leq i\leq n-2,\\[1mm]
[N_{(n-1)n},X^\alpha] & = a^{\alpha}_{(n-1)n,(n-1)n}N_{(n-1)n}+a^{\alpha}_{(n-1)n,1(n-1)}N_{1(n-1)},& \\[1mm] [N_{ij},X^\alpha] & = \sum\limits_{p=i}^{j-1}a^{\alpha}_{p(p+1),p(p+1)}N_{ij}, & i+1<j.
\end{array}$$

It is easy to see that $[X^\alpha,N_{12}]+[N_{12},X^\alpha]$
belongs to the right annihilator of the algebra of $L(n,f).$ From
the chain of equalities
$$0=[N_{12},[X^\alpha,N_{12}]+[N_{12},X^\alpha]]=
[N_{12},\sum\limits_{i=3}^{n-1}b^{\alpha}_{12,2i}N_{2i}+(a^{\alpha}_{12,2n}+b^{\alpha}_{12,2n})N_{2n}]=$$
$$=\sum\limits_{i=3}^{n-1}b^{\alpha}_{12,2i}N_{1i}+(a^{\alpha}_{12,2n}+b^{\alpha}_{12,2n})N_{1n},$$
we deduce $b^{\alpha}_{12,2j}=0, \ 3\leq j\leq n-1$ and $b^{\alpha}_{12,2n}=-a^{\alpha}_{12,2n}.$

Similarly, from
$$0=[N_{1i},[X^\alpha,N_{12}]+[N_{12},X^\alpha]]=[N_{1i},\sum\limits_{j=i+1}^nb^\alpha_{ij}N_{ij}]
=\sum\limits_{j=i+1}^nb^\alpha_{ij}N_{1j}, \quad i > 2,$$ we
derive $b^{\alpha}_{12,ij}=0, \ 2< i<j\leq n.$

From the equality
$$0=[N_{i(i+1)},[X^\alpha,N_{12}]+[N_{12},X^\alpha]], \quad i\geq2,$$
we get $$b^{\alpha}_{12,12}=-a^{\alpha}_{12,12}, \quad b^{\alpha}_{12,1i}=0, \ 3\leq i\leq n-1.$$

Therefore, we obtain $$[X^\alpha,N_{12}]=-a^{\alpha}_{12,12}N_{12}-a^{\alpha}_{12,2n}N_{2n}+b^{\alpha}_{12,1n}N_{1n}.$$

Applying analogous argumentations as we used above for the
products with  $k\geq2,$
$$\begin{array}{lll}
[N_{1k},[X^\alpha,N_{i(i+1)}]+[N_{i(i+1)},X^\alpha]], & [N_{i(i+1)},[X^\alpha,N_{i(i+1)}]+[N_{i(i+1)},X^\alpha]], & 2\leq i\leq n-2,\\[1mm]
[N_{1k},[X^\alpha,N_{(n-1)n}]+[N_{(n-1)n},X^\alpha]], & [N_{i(i+1)},[X^\alpha,N_{(n-1)n}]+[N_{(n-1)n},X^\alpha]],&\\[1mm]
[N_{1k},[X^\alpha,N_{ij}]+[N_{ij},X^\alpha]], & [N_{i(i+1)},[X^\alpha,N_{ij}]+[N_{ij},X^\alpha]], & 1<j-i<n-1,\\[1mm]
[N_{1k},[X^\alpha,N_{1n}]+[N_{1n},X^\alpha]], & [N_{i(i+1)},[X^\alpha,N_{1n}]+[N_{1n},X^\alpha]],&
\end{array}$$
we obtain
$$\begin{array}{lll}
[X^\alpha,N_{i(i+1)}] =-a^{\alpha}_{i(i+1),i(i+1)}N_{i(i+1)}+b^{\alpha}_{i(i+1),1n} N_{1n}, & 2\leq i\leq n-2,\\[1mm]
[X^\alpha,N_{(n-1)n}]=-a^{\alpha}_{(n-1)n,(n-1)n}N_{(n-1)n}-a^{\alpha}_{(n-1)n,1(n-1)}N_{1(n-1)}+
b^{\alpha}_{(n-1)n,1n}N_{1n}, & \\[1mm]
[X^\alpha,N_{ij}] =-\sum\limits_{p=i}^{j-1}a^{\alpha}_{p(p+1),p(p+1)}N_{ij}+b^{\alpha}_{ij,1n}N_{1n}, & 1<j-i<n-1,\\[1mm]
[X^\alpha,N_{1n}] =b^{\alpha}_{1n,1n}N_{1n}.
\end{array}$$

From the chain of equalities
$$[X^\alpha,N_{1n}]=[X^\alpha,[N_{12},N_{2n}]]=[[X^\alpha,N_{12}],N_{2n}]-[[X^\alpha,N_{2n}],N_{12}]=$$
$$[-a^{\alpha}_{12,12}N_{12}-a^{\alpha}_{12,2n}N_{2n}+b^{\alpha}_{12,1n}N_{1n},N_{2n}]-
[-\sum\limits_{p=2}^{n-1}a^{\alpha}_{p(p+1),p(p+1)}N_{2n}+b^{\alpha}_{2n,1n}N_{1n},N_{12}]=$$
$$-a^{\alpha}_{12,12}N_{1n}-\sum\limits_{p=2}^{n-1}a^{\alpha}_{p(p+1),p(p+1)}N_{1n}=-\sum\limits_{p=1}^{n-1}a^{\alpha}_{p(p+1),p(p+1)}N_{1n},$$
we get
$[X^\alpha,N_{1n}]=-\sum\limits_{p=1}^{n-1}a^{\alpha}_{p(p+1),p(p+1)}N_{1n}.$

By induction on $j$ we will prove
\begin{equation}\label{eq2}
[X^\alpha,N_{i(i+j)}]=-\sum\limits_{p=i}^{i+j-1}a^{\alpha}_{p(p+1),p(p+1)}N_{i(i+j)}, \quad  j-i\geq 2.
\end{equation}

The base of induction ensures the equalities
$$[X^\alpha,N_{i(i+2)}]=[X^\alpha,[N_{i(i+1)},N_{(i+1)(i+2)}]]=[[X^\alpha,N_{i(i+1)}],N_{(i+1)(i+2)}]-$$
$$[[X^\alpha,N_{(i+1)(i+2)}],N_{i(i+1)}]=-\sum\limits_{p=i}^{i+1}a^{\alpha}_{p(p+1),p(p+1)}N_{i(i+2)}, \quad  1\leq i\leq n-2.$$

Let us suppose that (\ref{eq2}) holds for $j$ and we will show it
for $j+1.$

For $i+j+1\leq n-1$ we have

$$[X^\alpha,N_{i(i+j+1)}]=[X^\alpha,[N_{i(i+j)},N_{(i+j)(i+j+1)}]]=[[X^\alpha,N_{i(i+j)}],N_{(i+j)(i+j+1)}]-$$
$$[[X^\alpha,N_{(i+j)(i+j+1)}],N_{i(i+j)}]=[-\sum\limits_{p=i}^{i+j-1}a^{\alpha}_{p(p+1),p(p+1)}N_{i(i+j)},
N_{(i+j)(i+j+1)}]-$$
$$[-a^{\alpha}_{(i+j)(i+j+1),(i+j)(i+j+1)}N_{(i+j)(i+j+1)}+b^{\alpha}_{(i+j)(i+j+1),1n} N_{1n},N_{i(i+j)}]=$$
$$-\sum\limits_{p=i}^{i+j}a^{\alpha}_{p(p+1),p(p+1)}N_{i(i+j+1)}.$$

The following chain of equalities complete the proof of equality (\ref{eq2})
$$[X^\alpha,N_{in}]=[X^\alpha,[N_{i(n-1)},N_{(n-1)n}]]=
[X^\alpha,N_{i(n-1)}],N_{(n-1)n}]-[X^\alpha,N_{(n-1)n}],N_{i(n-1)}]=$$
$$[-\sum\limits_{p=i}^{n-2}a^{\alpha}_{p(p+1),p(p+1)}N_{i(n-1)},N_{(n-1)n}]-[-a^{\alpha}_{(n-1)n,(n-1)n}N_{(n-1)n}-
a^{\alpha}_{(n-1)n,1(n-1)}N_{1(n-1)}+$$
$$b^{\alpha}_{(n-1)n,1n}N_{1n},N_{i(n-1)}]=-\sum\limits_{p=i}^{n-1}a^{\alpha}_{p(p+1),p(p+1)}N_{in}.$$

Therefore, we obtain
$$[X^\alpha,N_{12}]=-a^{\alpha}_{12,12}N_{12}-a^{\alpha}_{12,2n}N_{2n}+b^{\alpha}_{12,1n}N_{1n}.$$
$$[X^\alpha,N_{i(i+1)}]=-a^{\alpha}_{i(i+1),i(i+1)}N_{i(i+1)}+b^{\alpha}_{i(i+1),1n} N_{1n}, \ 2\leq i\leq n-2,$$
$$[X^\alpha,N_{(n-1)n}]=-a^{\alpha}_{(n-1)n,(n-1)n}N_{(n-1)n}-a^{\alpha}_{(n-1)n,1(n-1)}N_{1(n-1)}+b^{\alpha}_{(n-1)n,1n}N_{1n},$$
$$[X^\alpha,N_{ij}]=-\sum\limits_{p=i}^{j-1}a^{\alpha}_{p(p+1),p(p+1)}N_{ij}, \ j>i+1.$$

Comparison of the above products with notations in (\ref{eq1})
completes the proof of lemma.
\end{proof}

\begin{lem} \label{lem3.2}  For $1\leq \alpha, \beta \leq n$ we have $[X^\alpha,X^\beta]=\sigma^{\alpha\beta}N_{1n}$ for some $\sigma^{\alpha\beta}\in \mathbb{C}.$
\end{lem}
\begin{proof}

Consider
$$[N_{12},[X^\alpha,X^\beta]]=[[N_{12},X^\alpha],X^\beta]-[[N_{12},X^\beta],X^\alpha]=
[a^\alpha_{12,12}N_{12}+a^\alpha_{12,2n}N_{2n},X^\beta]-$$
$$[a^\beta_{12,12}N_{12}+a^\beta_{12,2n}N_{2n},X^\alpha]=
a^\alpha_{12,12}(a^\beta_{12,12}N_{12}+a^\beta_{12,2n}N_{2n})+
a^\alpha_{12,2n}(\sum\limits_{p=2}^{n-1}a^{\beta}_{p(p+1),p(p+1)}N_{2n})-$$
$$a^\beta_{12,12}(a^\alpha_{12,12}N_{12}+a^\alpha_{12,2n}N_{2n})-
a^\beta_{12,2n}(\sum\limits_{p=2}^{n-1}a^{\alpha}_{p(p+1),p(p+1)}N_{2n})=
(a^{\alpha}_{12,12}a^{\beta}_{12,2n}-a^{\beta}_{12,12}a^{\alpha}_{12,2n}-$$
$$\sum\limits_{p=2}^{n-1}a^{\alpha}_{p(p+1),p(p+1)}a^\beta_{12,2n}+\sum\limits_{p=2}^{n-1}a^{\beta}_{p(p+1),p(p+1)}a^\alpha_{12,2n})N_{2n}.$$

On the other hand,
$$[N_{12},[X^\alpha,X^\beta]]=[N_{12},\sum\limits_{1\leq
q-p<n}\sigma^{\alpha\beta}_{pq}
N_{pq}]=\sum\limits_{i=3}^n\sigma^{\alpha\beta}_{2i}N_{1i}.$$

Comparing coefficients at the basis elements we derive
$$\sigma^{\alpha\beta}_{2i}=0, \quad 3\leq i\leq n.$$


For $2\leq i\leq n-2$ we consider the chain of equalities
$$[N_{i(i+1)},[X^\alpha,X^\beta]]=[[N_{i(i+1)},X^\alpha],X^\beta]-[[N_{i(i+1)},X^\beta],X^\alpha]=$$
$$a^{\alpha}_{i(i+1),i(i+1)}(a^{\beta}_{i(i+1),i(i+1)}N_{i(i+1)}+a^{\beta}_{i(i+1),1n} N_{1n})+a^{\alpha}_{i(i+1),1n}\sum\limits_{p=1}^{n-1}a^{\beta}_{p(p+1),p(p+1)}N_{1n}-$$
$$a^{\beta}_{i(i+1),i(i+1)}(a^{\alpha}_{i(i+1),i(i+1)}N_{i(i+1)}+a^{\alpha}_{i(i+1),1n} N_{1n})-a^{\beta}_{i(i+1),1n}\sum\limits_{p=1}^{n-1}a^{\alpha}_{p(p+1),p(p+1)}N_{1n}=$$
$$(a^{\alpha}_{i(i+1),i(i+1)}a^{\beta}_{i(i+1),1n}+
a^{\alpha}_{i(i+1),1n}\sum\limits_{p=1}^{n-1}a^{\beta}_{p(p+1),p(p+1)}-$$
$$a^{\beta}_{i(i+1),i(i+1)}a^{\alpha}_{i(i+1),1n}-a^{\beta}_{i(i+1),1n}\sum\limits_{p=1}^{n-1}a^{\alpha}_{p(p+1),p(p+1)})N_{1n}.$$

On the other hand,
$$[N_{i(i+1)},[X^\alpha,X^\beta]]=
[N_{i(i+1)},\sum\limits_{k=1}^{i-1}\sigma^{\alpha\beta}_{ki}N_{ki}+
\sum\limits_{j=i+2}^n\sigma^{\alpha\beta}_{(i+1)j}N_{(i+1)j}]=-\sum\limits_{k=1}^{i-1}\sigma^{\alpha\beta}_{ki}N_{k(i+1)}+\sum\limits_{j=i+2}^n\sigma^{\alpha\beta}_{(i+1)j}N_{ij}.$$
Therefore,
$$\sigma^{\alpha\beta}_{ki}=\sigma^{\alpha\beta}_{js}=0, \ 1\leq k\leq i-1, \ 2\leq i\leq n-2, \ 3\leq j\leq n-1, \ j+1\leq s\leq n$$ and  $$[X^\alpha,X^\beta]=\sigma^{\alpha\beta}_{1(n-1)}N_{1(n-1)}+\sigma^{\alpha\beta}_{1n}N_{1n}.$$

Similar arguments for the products
$$[N_{(n-1)n},[X^\alpha,X^\beta]]$$
yield $\sigma^{\alpha\beta}_{1(n-1)}=0,$ which completes the proof
of the lemma. For convenience let us omit the lower indexes of
$\sigma^{\alpha\beta}_{1n}$.
\end{proof}

From Leibniz identity
$$[X^\alpha,[N_{i(i+1)},X^\alpha]]=[[X^\alpha,N_{i(i+1)}],X^\alpha]-[[X^\alpha,X^\alpha],N_{i(i+1)}]$$
for $1\leq i\leq n-1$ we obtain restrictions:
$$ a^{\alpha}_{i(i+1),i(i+1)}(a^{\alpha}_{i(i+1),1n}+b^{\alpha}_{i(i+1),1n})=0, \ 2\leq i\leq n-2,$$ $$a^{\alpha}_{12,12}b^{\alpha}_{12,1n}=a^{\alpha}_{(n-1)n,(n-1)n}b^{\alpha}_{(n-1)n,1n}=0.$$

Let us list again the obtained products between the basis
elements.  For $1\leq \alpha\leq f$ we have
$$\left\{\begin{array}{lll}
[N_{12},X^\alpha]=a^\alpha_{12,12}N_{12}+a^\alpha_{12,2n}N_{2n},&\\[1mm]
[N_{i(i+1)},X^\alpha]=a^{\alpha}_{i(i+1),i(i+1)}N_{i(i+1)}+a^{\alpha}_{i(i+1),1n} N_{1n}, & 2\leq i\leq n-2,\\[1mm]
[N_{(n-1)n},X^\alpha]=a^{\alpha}_{(n-1)n,(n-1)n}N_{(n-1)n}+a^{\alpha}_{(n-1)n,1(n-1)}N_{1(n-1)},&\\[1mm]
[N_{ij},X^\alpha]=\sum\limits_{p=i}^{j-1}a^{\alpha}_{p(p+1),p(p+1)}N_{ij}, & j>i+1,\\[1mm]
[X^\alpha,N_{12}]=-a^{\alpha}_{12,12}N_{12}-a^{\alpha}_{12,2n}N_{2n}+b^{\alpha}_{12,1n}N_{1n}, &\\[1mm]
[X^\alpha,N_{i(i+1)}]=-a^{\alpha}_{i(i+1),i(i+1)}N_{i(i+1)}+b^{\alpha}_{i(i+1),1n} N_{1n}, & 2\leq i\leq n-2,\\[1mm]
[X^\alpha,N_{(n-1)n}]=-a^{\alpha}_{(n-1)n,(n-1)n}N_{(n-1)n}-a^{\alpha}_{(n-1)n,1(n-1)}N_{1(n-1)}+
b^{\alpha}_{(n-1)n,1n}N_{1n},&\\[1mm]
[X^\alpha,N_{ij}]=-\sum\limits_{p=i}^{j-1}a^{\alpha}_{p(p+1),p(p+1)}N_{ij}, & j>i+1,\\[1mm]
[X^\alpha,X^\beta]=\sigma^{\alpha\beta}N_{1n},&
\end{array}\right.
$$
with restrictions on parameters:
$$a^{\alpha}_{i(i+1),i(i+1)}(a^{\alpha}_{i(i+1),1n}+b^{\alpha}_{i(i+1),1n})=0, \ 2\leq i\leq n-2,$$ $$a^{\alpha}_{12,12}b^{\alpha}_{12,1n}=a^{\alpha}_{(n-1)n,(n-1)n}b^{\alpha}_{(n-1)n,1n}=0.$$

Note that for solvable non-Lie Leibniz algebras of the set
$L(n,f)$ the following equality holds
\begin{equation}\label{eq3}
[X^\gamma,N_{1n}]=[N_{1n},X^\gamma]=0, \quad 1\leq \gamma\leq f.
\end{equation}

Indeed, if we assume the contrary, then taking into account that
$[X^\gamma,N_{1n}]=-[N_{1n},X^\gamma]$ we can assume
$[X^{\gamma},N_{1n}]\neq0$ for some $\gamma\in\{1, \dots, f\}.$

Simplifying the following products using Leibniz identity
$$[X^{\gamma},[N_{12},X^\alpha]+[X^\alpha,N_{12}]], \quad [X^{\gamma},[N_{i(i+1)},X^\alpha]+[X^\alpha,N_{i(i+1)}]],$$
$$[X^{\gamma},[N_{(n-1)n},X^\alpha]+[X^\alpha,N_{(n-1)n}]], \quad [X^{\gamma},[X^\alpha,X^\beta]+[X^\beta,X^\alpha]], \quad [X^{\gamma},[X^\alpha,X^\alpha]],$$
we obtain
$$b^{\alpha}_{12,1n}=b^{\alpha}_{(n-1)n,1n}=\sigma^{\alpha\alpha}=0, \quad b^{\alpha}_{i(i+1),1n}=-a^{\alpha}_{i(i+1),1n}, \quad
\sigma^{\alpha\beta}=-\sigma^{\beta\alpha}.$$ Thus we get a Lie
algebra, which is a contradiction.


\begin{cor} For a Leibniz algebra of the set $L(n,1)$ the matrices of the left and right operators $A=(a_{ij,pq}), \ B=(b_{ij,pq})$ have the following properties:

1) The maximum number of off-diagonal elements of matrix $A$ is $n-1;$

2) The maximum number of off-diagonal elements of matrix $B$ is $n+1.$
\end{cor}

\begin{thm} \label{thm3.4} Solvable Leibniz algebra of the set $L(n, n-1)$ is a Lie algebra.
\end{thm}
\begin{proof} Making suitable change of basis we can assume that operator
$R_{X^1}$ acts as follows
$$
\begin{array}{lll}
[N_{12},X^1]=N_{12}+a^1_{12,2n}N_{2n}, & & \\[1mm]
[N_{i(i+1)},X^1]=a^1_{i(i+1),1n} N_{1n}, & 2\leq i\leq n-2, & \\[1mm]
[N_{(n-1)n},X^1]=a^1_{(n-1)n,1(n-1)}N_{1(n-1)}, & & \\[1mm]
[N_{1j},X^1]=N_{1j}, & j>2. &
\end{array}
$$

Since $[N_{1n},X^1]=N_{1n},$ then from Equation (\ref{eq3}) it
follows that the algebra is a Lie algebra.
\end{proof}

So we present a description of solvable Leibniz algebras with
nilradical  $T(n)$. Moreover, in the case of maximal possible
dimension we show that this algebra is a Lie algebra.

\section{Illustration for low dimensions}

In this section we give the description of Leibniz algebras with nilradical $T(3)$ and $T(4)$.

Note that Lie algebra $T(3)$ is nothing else, but Heisenberg algebra $H(1)$. Solvable Leibniz algebras with Heisenberg nilradical were described in \cite{Lin}.

Therefore, we will consider case when $n=4.$ We know that the
complimentary vector space to nilradical $T(4)$ has dimension less
than four. In case when dimension of the complementary space is
equal to 3 we obtain a Lie algebra (see Theorem \ref{thm3.4}),
which falls into the classification already obtained in
\cite{Wint}. So we will consider dimension of the complimentary
vector space to be equal to 1 and 2.

\

\textbf{The Leinbiz algebras $L(4,1).$}

From previous section we have that the algebra $L(4,1)$ admits a
basis $\{N_{12}, N_{23}, N_{34}, N_{13},N_{24}, N_{14}, X\}$ in
which the table of multiplication has the following form:
\begin{equation}\label{eq4}
\left\{\begin{array}{llll}
[N_{12},X]&=&a_{12,12}N_{12}+a_{12,24} N_{24},&\\[1mm]
[X,N_{12}]&=&-a_{12,12}N_{12}-a_{12,24} N_{24}+b_{12,14}N_{14},&\\[1mm]
[N_{23},X]&=&a_{23,23}N_{23}+a_{23,14} N_{14},&\\[1mm]
[X,N_{23}]&=&-a_{23,23}N_{23}+b_{23,14}N_{14},&\\[1mm]
[N_{34},X]&=&-(a_{12,12}+a_{23,23})N_{34}+a_{34,13} N_{13},&\\[1mm] [X,N_{34}]&=&(a_{12,12}+a_{23,23})N_{34}-a_{34,13} N_{13}+b_{34,14}N_{14}, &\\[1mm] [N_{13},X]&=&-[X,N_{13}]=(a_{12,12}+a_{23,23})N_{13},&\\[1mm]
[N_{24},X]&=&-[X,N_{24}]=-a_{12,12}N_{24},&\\[1mm]
[X,X]&=&\sigma_{14}N_{14},&\\[1mm]
\end{array}\right.
\end{equation}
where
$$a_{12,12}b_{12,14}=a_{23,23}(a_{23,14}+b_{23,14})=(a_{12,12}+a_{23,23})b_{34,14}=0.$$

Since $L(4,1)$ is a non-nilpotent Leibniz algebra we have
$(a_{12,12},a_{23,23})\neq(0,0).$

{\bf Case 1.} Let $a_{12,12}=0$. Then $a_{23,23}\neq0, b_{23,14}=-a_{23,14}$ and $b_{34,14}=0$.

Taking the change of basis as follows:
$$X'=\frac{1}{a_{23,23}}X, \quad N_{23}'=N_{23}+\frac{a_{23,14}}{a_{23,23}}N_{14}, \quad N_{34}'=N_{34}-\frac{a_{34,13}}{2a_{23,23}}N_{13}$$ the multiplication (\ref{eq4}) transforms into
$$
\begin{array}{lll}
[N_{12},X]=a_{12,24} N_{24}, & [X,N_{12}]=-a_{12,24} N_{24}+b_{12,14}N_{14}, & \\[1mm]
[N_{23},X]=-[X,N_{23}]=N_{23}, & [N_{34},X]=-[X,N_{34}]=-N_{34}, & \\[1mm]
[N_{13},X]=-[X,N_{13}]=N_{13}, &  [X,X]=\sigma_{14}N_{14},
\end{array}
$$
where $(b_{12,14},\sigma_{14}) \neq (0, 0).$


{\bf Case 2.} Let $a_{12,12}\neq0,$ then $b_{12,14}=0.$ Taking the
change of basis $X'=\frac{1}{a_{12,12}}X,$  we can assume
$a_{12,12}=1.$

{\bf Subcase 2.1.} Let $a_{23,23}= 0.$ Then  $b_{34,14}=0.$

Applying a change of a basis
$$N_{12}'=N_{12}+\frac{a_{12,24}}{2}N_{24}, \quad
N_{34}'=N_{34}-\frac{a_{34,13}}{2}N_{13}$$ the products
(\ref{eq4})  simplify to the following:
$$\begin{array}{lll}
[N_{12},X]=-[X,N_{12}]=N_{12}, & [N_{34},X]=-[X,N_{34}]=-N_{34}, &  \\[1mm]
[N_{13},X]=-[X,N_{13}]=N_{13}, & [N_{24},X]=-[X,N_{24}]=-N_{24}, &\\[1mm]
[N_{23},X]=a_{23,14} N_{14}, & [X,N_{23}]=b_{23,14}N_{14}, &   \\[1mm]
[X,X]=\sigma_{14}N_{14}, &&
\end{array}$$
where $(a_{23,14}+b_{23,14}, \sigma_{14} ) \neq (0, 0).$

{\bf Subcase 2.2.} Let $a_{23,23}\neq0.$ Then $b_{23,14}=-a_{23,14}.$

{\bf Subcase 2.2.1.} Let $a_{23,23}=-1.$ Then substituting
$$N_{23}'=N_{23}-a_{23,14}N_{14}, \quad N_{12}'=N_{12}+\frac{a_{12,24}}{2}N_{24}$$
we derive to an algebra with the following table of
multiplication:
$$\begin{array}{lll}
[N_{12},X]=-[X,N_{12}]=N_{12}, & [N_{23},X]=[X,N_{23}]=-N_{23}, & \\[1mm]
[N_{34},X]=a_{34,13} N_{13}, & [X,N_{34}]=-a_{34,13} N_{13}+b_{34,14}N_{14}, \\[1mm]
[N_{24},X]=-[X,N_{24}]=-N_{24}, & [X,X]=\sigma_{14}N_{14}
\end{array}$$
where $(b_{12,14}, \sigma_{14}) \neq(0, 0).$

Note that by permuting the indexes of the basis elements of the
above algebra one obtains an algebra from Case 1.


{\bf Subcase 2.2.2.} Let $a_{23,23}\neq-1.$ Then $b_{34,14}=0.$

Setting
$$N_{12}'=N_{12}+\frac{a_{12,24}}{2}N_{24}, \quad N_{23}'=N_{23}+\frac{a_{23,14}}{a_{23,23}}N_{14},$$
$$ N_{34}'=\sigma_{14}(N_{34}-\frac{a_{34,13}}{2(1+a_{23,23})}N_{13}), \quad N_{24}'=\sigma_{14}N_{24}, \quad N_{14}'=\sigma_{14}N_{14}$$
we get an algebra with the following table of multiplications:
$$
\begin{array}{lll}
[N_{12},X]=-[X,N_{12}]=N_{12}, & [N_{23},X]=-[X,N_{23}]=a_{23,23}N_{23}, \\[1mm]
[N_{34},X]=-[X,N_{34}]=-(1+a_{23,23})N_{34}, & [N_{13},X]=-[X,N_{13}]=(1+a_{23,23})N_{13}, \\[1mm]  [N_{24},X]=-[X,N_{24}]=-N_{24}, & [X,X]=N_{14},
\end{array}
$$
where $(1+a_{23,23})a_{23,23}\neq0.$

Non-isomorphisms of obtained algebras can be easily established by considering the dimensions of derived series of the algebras.

Thus, the following theorem is proved.

\begin{thm} An arbitrary non-Lie Leibniz algebra of the set $L(4,1)$ is isomorphic to one of the following pairwise non-isomorphic algebras:
$$L_1:
\begin{array}{lll}
[N_{12},X]=a_{12,24} N_{24}, & [X,N_{12}]=-a_{12,24} N_{24}+b_{12,14}N_{14}, & \\[1mm]
[N_{23},X]=-[X,N_{23}]=N_{23}, & [N_{34},X]=-[X,N_{34}]=-N_{34}, & \\[1mm]
[N_{13},X]=-[X,N_{13}]=N_{13}, &  [X,X]=\sigma_{14}N_{14},
\end{array}
$$
where $(b_{12,14},\sigma_{14})\neq(0,0).$

$$L_2:
\begin{array}{lll}
[N_{12},X]=-[X,N_{12}]=N_{12}, & [N_{34},X]=-[X,N_{34}]=-N_{34} &  \\[1mm]
[N_{13},X]=-[X,N_{13}]=N_{13}, & [N_{24},X]=-[X,N_{24}]=-N_{24}, &\\[1mm]
[N_{23},X]=a_{23,14} N_{14}, & [X,N_{23}]=b_{23,14}N_{14}, &   \\[1mm]
[X,X]=\sigma_{14}N_{14}, &&
\end{array}$$
where $(a_{23,14}+b_{23,14},\sigma_{14})\neq(0,0).$

$$L_3:
\begin{array}{lll}
[N_{12},X]=-[X,N_{12}]=N_{12}, & [N_{23},X]=-[X,N_{23}]=a_{23,23}N_{23}, \\[1mm]
[N_{34},X]=-[X,N_{34}]=-(1+a_{23,23})N_{34}, & [N_{13},X]=-[X,N_{13}]=(1+a_{23,23})N_{13}, \\[1mm]  [N_{24},X]=-[X,N_{24}]=-N_{24}, & [X,X]=N_{14}.
\end{array}
$$
where $(1+a_{23,23})a_{23,23}\neq0.$
\end{thm}

\textbf{The Leibniz algebras $L(4,2).$}

Classification of Leibniz algebras in this set is presented in the
following theorem.
\begin{thm} An arbitrary non-Lie Leibniz algebra of the set $L(4,2)$ admits a basis $\{N_{12}, N_{23}, N_{34}, N_{13}, N_{24}, N_{14}, X^1,X^2\}$ in which the table of multiplication has the following form:
$$[N_{12},X^1]=-[X^1,N_{12}]=N_{12}, \quad [N_{34},X^1]=-[X^1,N_{34}]=-N_{34},$$
$$[N_{13},X^1]=-[X^1,N_{13}]=N_{13}, \quad [N_{24},X^1]=-[X^1,N_{24}]=-N_{24},$$
$$[N_{23},X^2]=-[X^2,N_{23}]=N_{23}, \quad [N_{34},X^2]=-[X^2,N_{34}]=-N_{34},$$
$$[N_{13},X^2]=-[X^2,N_{13}]=N_{13}, \quad [X^1,X^1]=\sigma^{11}N_{14},$$
$$ [X^2,X^2]=\sigma^{22}N_{14}, \quad [X^1,X^2]=\sigma^{12}N_{14},  \quad [X^2,X^1]=\sigma^{21}N_{14}.$$
\end{thm}
\begin{proof}
From Lemmas \ref{lem3.1} and \ref{lem3.2} we have
$$\begin{array}{ll}
[N_{12},X^1]=a^1_{12,12}N_{12}+a^1_{12,24} N_{24}, & \\[1mm]
[X^1,N_{12}]=-a_{12,12}^1N_{12}-a^1_{12,24} N_{24}+b^1_{12,14}N_{14}, & \\[1mm]
[N_{23},X^1]=a^1_{23,23}N_{23}+a^1_{23,14} N_{14},  & \\[1mm]
[X^1,N_{23}]=-a^1_{23,23}N_{23}+b^1_{23,14}N_{14}, & \\[1mm]
[N_{34},X^1]=-(a^1_{12,12}+a^1_{23,23})N_{34}+a^1_{34,13} N_{13}, & \\[1mm]
[X^1,N_{34}]=(a^1_{12,12}+a^1_{23,23})N_{34}-a^1_{34,13} N_{13}+b^1_{34,14}N_{14}, & \\[1mm]
[N_{13},X^1]=-[X^1,N_{13}]=(a^1_{12,12}+a^1_{23,23})N_{13}, & \\[1mm]
[N_{24},X^1]=-[X^1,N_{24}]=-a^1_{12,12}N_{24}, & \\[1mm]
[N_{12},X^2]=a^2_{12,12}N_{12}+a^2_{12,24} N_{24},  & \\[1mm]
[X^2,N_{12}]=-a_{12,12}^2N_{12}-a^2_{12,24} N_{24}+b^2_{12,14}N_{14}, & \\[1mm]
[N_{23},X^2]=a^2_{23,23}N_{23}+a^2_{23,14} N_{14}, & \\[1mm]
[X^2,N_{23}]=-a^2_{23,23}N_{23}+b^2_{23,14}N_{14}, & \\[1mm]
[N_{34},X^2]=-(a^2_{12,12}+a^2_{23,23})N_{34}+a^2_{34,13} N_{13}, & \\[1mm]
[X^2,N_{34}]=(a^2_{12,12}+a^2_{23,23})N_{34}-a^2_{34,13} N_{13}+b^2_{34,14}N_{14}, & \\[1mm]
[N_{13},X^2]=-[X^2,N_{13}]=(a^2_{12,12}+a^2_{23,23})N_{13}, & \\[1mm]
[N_{24},X^2]=-[X^2,N_{24}]=-a^2_{12,12}N_{24}
\end{array}$$
with the restrictions
$$a_{12,12}^1b_{12,14}^1=a_{23,23}^1(a_{23,14}^1+b_{23,14}^1)=(a^1_{12,12}+a^1_{23,23})b_{34,14}^1=0,$$
$$a_{12,12}^2b_{12,14}^2=a_{23,23}^2(a_{23,14}^2+b_{23,14}^2)=(a^2_{12,12}+a^2_{23,23})b_{34,14}^2=0.$$

Taking the change of basis
$$X^{1'}=\frac{a^2_{23,23}}{a^1_{12,12}a^2_{23,23}-
a^2_{12,12}a^1_{23,23}}X^1-\frac{a^1_{23,23}}{a^1_{12,12}a^2_{23,23}-a^2_{12,12}a^1_{23,23}}X^2,$$
$$X^{2'}=-\frac{a^2_{12,12}}{a^1_{12,12}a^2_{23,23}-
a^2_{12,12}a^1_{23,23}}X^1+\frac{a^1_{12,12}}{a^1_{12,12}a^2_{23,23}-a^2_{12,12}a^1_{23,23}}X^2,$$
we deduce
$$\begin{array}{lll}
[N_{12},X^1]=-[X^1,N_{12}]=N_{12}+a^1_{12,24} N_{24}, & [N_{23},X^1]=a^1_{23,14} N_{14}, & \\[1mm]
[X^1,N_{23}]=b^1_{23,14}N_{14}, & [N_{34},X^1]=-[X^1,N_{34}]=-N_{34}+a^1_{34,13} N_{13}, & \\[1mm]
[N_{13},X^1]=-[X^1,N_{13}]=N_{13}, & [N_{24},X^1]=-[X^1,N_{24}]=-N_{24}, & \\[1mm]
[N_{12},X^2]=a^2_{12,24} N_{24},  & [X^2,N_{12}]=-a^2_{12,24} N_{24}+b^2_{12,14}N_{14}, & \\[1mm]
[N_{23},X^2]=-[X^2,N_{23}]=N_{23}+a^2_{23,14} N_{14}, & [N_{34},X^2]=-[X^2,N_{34}]=-N_{34}+a^2_{34,13} N_{13}, & \\[1mm]
[N_{13},X^2]=-[X^2,N_{13}]=N_{13}. & &
\end{array}$$

Applying Leibniz identity for the following triples of elements:
 $$(N_{12},X^1,X^2), \ (N_{23},X^1,X^2), \ (N_{34},X^1,X^2), \ (X^1,N_{23},X^2), \ (X^2,N_{12},X^1)$$ we get $$ a_{12,24}^2=a_{23,14}^1=a_{34,13}^1=a_{34,13}^2=b_{23,14}^1=b_{12,14}^2=0.$$

Finally, taking the basis transformation:
$$N_{12}'=N_{12}+\frac{a_{12,24}^1}{2}N_{24}, \ \
N_{23}'=N_{23}+a_{23,14}^2N_{14}$$ we obtain the table of
multiplication listed in the assertion of theorem.
\end{proof}

\end{document}